\documentclass[11 pt]{amsart}
\usepackage{amscd,amsfonts,amssymb,amsmath,tikz-cd}
\usepackage{hyperref}
\usepackage{epsfig}
\usepackage[shortlabels]{enumitem}
\usepackage[nocompress]{cite}
\newtheorem{theorem}{Theorem}[section]
\newtheorem{corollary}[theorem]{Corollary}
\newtheorem{lemma}[theorem]{Lemma}
\newtheorem{proposition}[theorem]{Proposition}

\theoremstyle{definition}

\newtheorem{question}[theorem]{Question}

\newtheorem{definition}[theorem]{Definition}
\newtheorem{example}[theorem]{Example}
\newtheorem{examples}[theorem]{Examples}
\newtheorem{remark}[theorem]{Remark}
\numberwithin{equation}{subsection}
\newtheorem*{ack}{Acknowledgement}
\usepackage[all,cmtip]{xy}
\newcommand{\Ker}{\operatorname{Ker}}

\newcommand{\im}{\operatorname{Im}}

\newcommand{\Aut}{\operatorname{Aut}}
\newcommand{\End}{\operatorname{End}}

\newcommand{\Conj}{\operatorname{Conj}}

\newcommand{\Core}{\operatorname{Core}}

\newcommand{\Inn}{\operatorname{Inn}}

\newcommand{\Hom}{\operatorname{Hom}}

\newcommand{\id}{\mathrm{id}}

\newcommand{\overbar}[1]{\mkern 1.5mu\overline{\mkern-1.5mu#1\mkern-1.5mu}\mkern 1.5mu}
\usepackage[autostyle]{csquotes}

\begin{document}
\title{Endomorphism monoid and automorphism group of residually finite and profinite quandles}

\author{Manpreet Singh}

\address{Deaprtment of mathematics and statistics, University of south florida, tampa, FL, 33620}
\email{manpreet.math23@gmail.com}


\begin{abstract}
We explore residually finite and profinite quandles. We prove that the endomorphism monoid and the automorphism group of finitely generated residually finite quandles are residually finite. In fact, we establish the similar result for a broad class of residually finite quandles. We provide a topological characterization of profinite quandles. We establish necessary and sufficient conditions for profinite quandles ensuring that their endomorphism monoids and automorphism groups are profinite.
\end{abstract}

\maketitle
\markboth{Manpreet Singh}{Endomorphism monoid and automorphism group of residually finite and profinite quandles}

\section*{Introduction}

The collection of all symmetries of a space forms a group, however, a subset of symmetries may not constitute a group. For example, the collection of reflection symmetries of a regular polygon does not form a group; instead, it forms a non-associative algebraic structure know as quandles. Takasaki  \cite{Takasaki} first studied quandles to investigate the reflection symmetries in finite geometry. For more on symmetric aspect of quandles, we recommend referring to \cite{MR4612506}.

In the 1980s, Joyce \cite{MR0638121} and Matveev \cite{MR0672410} independently rediscovered quandles, recognizing their fundamental role in the study of knots. They proved that every knot is completely determined by its knot quandle up to the orientation of the space and the knot itself. The axioms of quandles are algebraic interpretation of the Reidemeister moves on knot diagrams. Since then, quandles are studied extensively to construct new knot invariants.

In \cite{MR3981139,MR4075375}, the study examined the residual finiteness property of quandles and proved that all link quandles are residually finite. Additionally, residual finiteness of the automorphism group was established for certain classes of residually finite quandles.

Recently, there has been extensive study of the profinite completion of the fundamental group of $3$-manifolds to investigate the topological properties and invariants of $3$-manifolds (as seen in \cite{MR3848406,MR4205634}). In \cite{MR3848406}, it was proven that the Alexander polynomial of a knot is determined by the profinite completion of its knot group.

In this article, we explore residually finite quandles, profinite quandles, their topological aspects, and the associated endomorphism monoid and automorphism group. The article is structured as follows:

In Section \ref{sec:preliminaries}, we recall some key definitions in quandle theory. In Section \ref{sec:endo_of_residually_finite_quandle}, we prove that if a quandle $Q$ is residually finite and has only a finite number of congruences for each $ n \geq 1$, then both the endomorphism monoid $\End(Q)$ and the automorphism group $\Aut(Q)$ are also residually finite, and $Q$ is Hopfian. As a result, for every finitely generated residually finite quandle $Q$, both $\End(Q)$ and $\Aut(Q)$ exhibit the property of being residually finite.

In Section \ref{sec:profinite_quandles}, we present some general results about profinite quandles and provide examples. We give a necessary and sufficient condition for a quandle to be residually finite in terms of profinite quandles. Additionally, we give an example of a countably infinite profinite quandle, which is in contrast to the situation in profinite groups.

In Section \ref{sec:topological_characterization_of_profinite_quandles}, we study the profinite quandles from a topological aspect. We prove that for a given projective system of compact Hausdorff quandles $(Q_i, \varphi_{ij}, I)$, the projective limit $\underset{i \in I}{\varprojlim}~Q_i$ exists (Theorem \ref{thm:existence_of_projective_limits}). Notably, this applies to a projective system of finite quandles with the discrete topology as well. Furthermore, in Theorem \ref{thm:topological_charactersization_of_profinite_quandles}, we present a topological characterization of profinite quandles. While it is well-known that a compact, Hausdorff, and totally disconnected group (or semigroup) is profinite, the situation with quandles remains unclear (see Question \ref{ques:converse?}).

In Section \ref{sec:endo_of_profinite_monoid}, we examine the profinite property of a topological quandles. In Theorem \ref{thm:when_does_end_and_aut_are_profinite}, we prove that given a compact Hausdorff and totally disconnected quandle $Q$, if the endomorphism monoid $\End(Q)$ (respectively, the automorphism group $\Aut(Q)$) is compact within the compact-open topology, then it is a profinite monoid (respectively, profinite group). Theorem \ref{thm:compactness_of_end_and_aut_in_profinite_quandle} provides a necessary and sufficient condition for a profinite quandle $Q$ so that $\End(Q)$ (respectively, $\Aut(Q)$) is profinite. Furthermore, in Theorem \ref{thm:end_aut_inn_as_projective_limits}, we show that for a profinite quandle $Q$, $\End(Q)$ and $\Aut(Q)$ can be viewed as projective limits of finite monoids and finite groups, which subsequently implies that the inner automorphism group $\Inn(Q)$ is profinite. The proofs in this section follow a similar approach to those presented in the paper \cite{MR2849853}.

\section{Preliminaries}\label{sec:preliminaries}

A {\it quandle} is a non-empty set $Q$ together with a binary operation $*$ satisfying the following axioms:
\begin{enumerate}[label={\textbf{Q}\arabic*}]
\item  $x*x=x$\, for all $x\in Q$.
\item  For each $x, y \in Q$, there exists a unique $z \in Q$ such that $x=z*y$.
\item  $(x*y)*z=(x*z)*(y*z)$\, for all $x,y,z\in Q$.
\end{enumerate}

The axiom {\textbf Q2} is equivalent to saying that the right multiplication by each element of $Q$ is a bijection. This gives a dual binary operation $*^{-1}$ on $Q$ defined as $x*^{-1}y=z$ if $x=z*y$.
\par

\begin{examples}
The following are some typical examples of quandles.
\begin{itemize}
\item Let $X$ be a non-empty set with binary operation $x*y=y$ for all $x, y \in X$. Then $(X, *)$ is a quandle and is termed a {\it trivial quandle}. The trivial quandle on $n$ elements is denoted by $T_n$.
\item Given a  group $G$ and an integer $n$, defining the binary operation $x*y=y^{-n}xy^n$ turns $G$ into a quandle referred to as the $n$-\textit{conjugation quandle} of $G$, and is denoted by $\Conj_n(G)$.
\item Given a group $G$, defining the binary operation $x*y=yx^{-1}y$ turns $G$ into a quandle referred to as the \textit{core quandle} of $G$, and is denoted by $\Core(G)$. 
In particular, if $G$ is a cyclic group of order $n$, then it is called the \textit{dihedral quandle}.
\item Given a group $G$ and an automorphism $\psi \in \Aut(G)$, there is a quandle structure on $G$ given by $x *y= \psi(x y^{-1})y$, called an {\it Alexander quandle}.
\end{itemize}
\end{examples}

A {\it congruence} on a quandle $Q$ with a binary operation $*$ is an equivalence relation on $Q$ that is compatible with the operations $*$ and $*^{-1}$ on $Q$, that is, for a congruence  $\alpha$ on $Q$ and $a,b,c,d \in Q$, if $a, b$ are in the same class of $\alpha$ and $c, d$ are in the same class of $\alpha$, then $a*^{\epsilon}c, b *^{\epsilon}d$ are in the same class of $\alpha$, where $\epsilon=\pm 1$. Let $Q$ and $R$ be two quandles. Every homomorphism $f: Q \to R$ gives rise to a congruence on $Q$ which is the kernel of map $f$ given by
$$
\ker(f) = \{ (a,b)~:~ f(a)=f(b)\}.
$$
The first isomorphism theorem implies that $Q/\ker(f) \cong \im(f)$, that is, congruences on $Q$ and homomorphic images of $Q$ are equivalent. For a congruence $\alpha$ on $Q$ and $a\in Q$, $[a]_{\alpha}$ denotes the equivalence class of $a$. For $b \in Q$, define $[a]_{\alpha} * b = \{ q*b~:~ q \in [a]_{\alpha}\}$.

A congruence $\alpha$ on $Q$ is said to be a {\it finite index congruence} if $Q/\alpha$ is a finite quandle, and the {\it index} of $\alpha$ is the cardinality of $Q/\alpha$.

Let $Q$ be a quandle and $\alpha$ a congruence on $Q$. Then $\alpha$ is said to be {\it characteristic} if for $(x,y) \in \alpha$ implies $(f(x), f(y)) \in \alpha$ for all $f \in \Aut(Q)$, where $\Aut(Q)$ denotes the set of all automorphisms of $Q$. Moreover, $\alpha$ is said to be {\it fully invariant} if for any $(x,y) \in \alpha$, it holds that $(f(x), f(y)) \in \alpha$ for every $f\in \End(Q)$, where $\End(Q)$ denotes the set of all endomorphisms of $Q$.

\section{Endomorphism monoid of a residually finite quandle}\label{sec:endo_of_residually_finite_quandle}

A quandle $Q$ is termed {\it residually finite}, if for any distinct elements $a,b \in Q$, there exists a finite quandle $F$ and a homomorphism $f: Q \to F$ such that $f(a) \neq f(b)$. For a residually finite group $G$, both $\Conj_n(G)$ and $\Core(G)$ are residually finite quandles. In \cite{MR3981139, MR4075375}, it is proved that all link quandles are residually finite.

It is a well known fact that the automorphism group of a finitely generated residually finite group is residually finite. In \cite[Proposition 4.4, Proposition 4.5]{MR3981139}, residual finiteness of automorphism group of a certain class of residually finite quandles were studied.

\begin{lemma}\label{lemma:existence-of-fully-invariant-finite-index-congruence}
Let $Q$ be a quandle. Suppose $Q$ admits only finitely many congruences of index $n$ for each $n \geq 1$. Then for each finite index congruence $\gamma$, there exists a fully invariant finite index congruence $\bar{\gamma}$ such that $\bar{\gamma} \subseteq \gamma$.
\end{lemma}
\begin{proof}
Let $\gamma$ be a congruence on $Q$ of index $n$ and $\mathcal{A}_n$ be the collection of all congruences on $Q$ of index at most $n$. Let $\bar{\gamma}=\underset{\alpha \in \mathcal{A}_n}{\cap}\alpha$. Then since $\mathcal{A}_n$ is a finite set, $\bar{\gamma}$ is a finite index congruence. We claim that $\bar{\gamma}$ is fully invariant. For any $\beta \in \mathcal{A}_n$ and $f \in \End(Q)$, note that $(f \times f)^{-1}(\beta)=\ker(\pi \circ f)$ and is of index at most $n$, where $\pi: Q \to Q/\bar{\gamma}$ is the natural homomorphism.
Now if $(a,b) \in \bar{\gamma}$, then $(a,b) \in \alpha$ for all $\alpha \in \mathcal{A}_n$, which implies that $(a,b) \in (f(a), f(b)) \in \alpha$ for all $\alpha \in \mathcal{A}_n$ and for all $f\in \End(Q)$, and hence $(f(a), f(b)) \in \bar{\gamma}$.
\end{proof}

A quandle $Q$ is said to be {\it Hopfian} if every onto quandle endomorphism of $Q$ is automorphism.
\begin{proposition}\label{prop:Hopfian_residually_finite_quandle}
Let $Q$ be a residually finite quandle. If for each $n \geq 1$, $Q$ has only finitely many congruences of index $n$, then $Q$ is Hopfian.
\end{proposition}
\begin{proof}
Let $\phi: Q \to Q$ be an epimorphism. We claim that $\phi$ is injective. To the contrary, suppose that there exist $p \neq q \in Q$ with $\phi(p) = \phi(q)$. Given that $Q$ is residually finite, there exists a finite index congruence $\alpha$ on $Q$ such that $(p,q) \notin \alpha$. According to Lemma \ref{lemma:existence-of-fully-invariant-finite-index-congruence}, there exists a fully invariant finite index congruence $\bar{\alpha}$ on $Q$ such that $\bar{\alpha} \subseteq \alpha$. Consequently, there is an onto homomorphism $\bar{\phi}: Q/\bar{\alpha} \to Q/\bar{\alpha}$. Since $\bar{\alpha}$ is of finite index, it follows that $\bar{\phi}$ is an automorphism. Because $\phi(p) =\phi(q)$, it follows that $\bar{\phi}([p]_{\bar{\alpha}})=\bar{\phi}([q]_{\bar{\alpha}})$, which leads to a contradiction because $[p]_{\bar{\alpha}}\neq [q]_{\bar{\alpha}}$. Thus $\phi$ is an automorphism and so $Q$ is Hopfian.
\end{proof}

\begin{theorem}\label{thm:residually-finite-endomorphism}
Let $Q$ be a residually finite quandle. Suppose that, for each $n \geq 1$, $Q$ admits only finitely many  congruences of index $n$. Then $\End(Q)$ is a residually finite monoid and $\Aut(Q)$ is a residually finite group.
\end{theorem}
\begin{proof}
Let $f,g \in \End(Q)$ such that $f\neq g$. Then there exists $q_0 \in Q$ such that $f(q_0)\neq g(q_0)$. As $Q$ is residually finite, there exists a finite index congruence $\alpha$ on $Q$ such that $(f(q_0), g(q_0)) \notin \alpha$. According to Lemma \ref{lemma:existence-of-fully-invariant-finite-index-congruence} there exists a fully invariant finite index congruence $\bar{\alpha}$ on $Q$ such that $\bar{\alpha} \subseteq \alpha$.
Thus there is a monoid homomorphism from $\End(Q)$ to $\End(Q/\bar{\alpha})$, defined as $h \mapsto \bar{h}$, where $\bar{h}([q]_{\bar{\alpha}})=[h(q)]_{\bar{\alpha}}$. Now since $(f(q_0), h(q_0)) \notin \alpha$, thus $(f(q_0), h(q_0)) \notin \bar{\alpha}$, which implies that $\bar{f} \neq \bar{g}$. Given the fact that $Q/\bar{\alpha}$ is a finite quandle, we deduce that $\End(Q)$ is a residually finite monoid.
Since $\Aut(Q)$ is a submonoid of $\End(Q)$, thus it is a residually finite monoid and hence a residually finite group.
\end{proof}

\begin{corollary}
Let $Q$ be a finitely generated quandle. Then $\End(Q)$ is a residually finite monoid and $\Aut(Q)$ is a residually finite group.
\end{corollary}
\begin{proof}

Each finite quandle consisting of $n$ elements can be depicted by an integral $n \times n$ matrix (see \cite{MR2175299}). Since there are only finitely many such matrices, it follows that there are only finitely many quandles of order $n$. 
The quandle $Q$ being finitely generated implies that there are only a finite number of homomorphisms from $Q$ to any finite quandle $F$. Consequently, $Q$ admits only a finite number of open congruences of index $n$ for $n\geq 1$. Now the proof follows from Theorem \ref{thm:residually-finite-endomorphism}.
\end{proof}

An infinite trivial quandle $T$ is residually finite but it automorphism group $\Aut(T)$ is the symmetric group on $T$, which is not residually finite. Therefore the condition in Theorem \ref{thm:residually-finite-endomorphism} is not redundant.

\section{Profinite quandles}\label{sec:profinite_quandles}
Given a quandle $Q$, we denote the collection of congruences on $Q$ by $\Ker(Q)$ and the collection of finite index congruences on $Q$ by $\underset{< \infty}{\Ker(Q)}$. Moreover, let $\mathcal{K}(Q)= \underset{{\alpha \in \ker(Q)}}{\cap} \alpha$ and $\underset{< \infty}{\mathcal{K}(Q)}=\underset{{\alpha \in \underset{< \infty}{\Ker}(Q)}}{\cap} \alpha.$ Note that the diagonal on the set $Q \times Q$ is a subset of both $\mathcal{K}(Q)$ and $\underset{< \infty}{\mathcal{K}(Q)}$. The set $\underset{< \infty}{\mathcal{K}(Q)}$ is called the {\it profinite kernel} of $Q$.
\begin{proposition}\label{pro:1.6}
A quandle $Q$ is residually finite if and only if  $\underset{< \infty}{\mathcal{K}(Q)}$ is the diagonal on $Q \times Q$. 
\end{proposition}
\begin{proof}
Let $(q_1, q_2) \in \underset{< \infty}{\mathcal{K}(Q)}$ such that $q_1 \neq q_2$. Then $(q_1, q_2) \in \alpha$ for all $\alpha \in \underset{< \infty}{\mathcal{K}(Q)}$ which implies that $Q$ is not residually finite. Conversely, if $\underset{< \infty}{\mathcal{K}(Q)}$ is the diagonal on $Q \times Q$, then for each $(q_1, q_2) \in Q \times Q$ such that $q_1 \neq q_2$, there must exists $\alpha \in \underset{< \infty}{\Ker} $ such that $(q_1,q_2) \notin \alpha.$ Thus $Q$ is residually finite.
\end{proof}

A set $I$ equipped with a binary relation $\leq$ is said to be {\it partially ordered} set if $\leq$ is both reflexive and transitive. A {\it directed set} is a partially ordered set $I$ satisfying the following condition: for all $i, j \in I$, there always exists an element $k \in I$ such that $i \leq k$ and $j \leq k$.

Let $(I,\leq)$ be a directed set. A {\it projective system} of quandles over $I$ consists of the following data:
\begin{enumerate}
\item a family of quandles $(Q_i)_{i \in I}$ indexed by $I$,
\item for each $i,j \in I$ such that $i \leq j$, a homomorphism $\varphi_{ij}: Q_j \to Q_i$ satisfying the following conditions:
\begin{enumerate}
\item[(2.1)] $\varphi_{ii}=\id_{Q_i}$ (identity map on $Q_i$) for all $i \in I$,
\item[(2.2)] $\varphi_{ij} \circ \varphi_{jk} = \varphi_{ik}$ for all $i,j,k \in I$ such that $i\leq j\leq k$.
\end{enumerate}
\end{enumerate}
The above projective system is denoted by $(Q_i, \varphi_{ij}, I)$.
\par

Consider a projective system of quandles $(Q_i, \varphi_{ij}, I)$ and a quandle $Q$. Then, a family of homomorphisms $\varphi_i: Q \to Q_i$, mapping $Q$ into each $Q_i$, is termed  {\it compatible}, if for all $i \leq j$, it holds that $\varphi_{ij} \circ \varphi_j = \varphi_i$, that is, the following diagram commutes:

\begin{center}

\begin{tikzcd}
                                & Q \arrow[ld, "\varphi_j"'] \arrow[rd, "\varphi_i"] &     \\
Q_j \arrow[rr, "\varphi_{ij}"'] &                                                    & Q_i
\end{tikzcd}

\end{center}

A {\it projective limit} of the projective system $(Q_i, \varphi_{ij}, I)$ of quandles is a quandle $Q=\underset{i \in I}{\varprojlim} ~Q_i$ along with a compatible family of homomorphisms $\varphi_i: Q \to Q_i$ satisfying the following universal property: for any quandle $R$ and any compatible family of homomorphisms $r_i:R \to Q_i$, there exists a unique homomorphism $\theta: R \to Q$ such that $\varphi_i \circ \theta = r_i$ for all $i \in I$.

Consider a projective system of quandles $(Q_i, \varphi_{ij}, I)$, and let $P= \prod_{i \in I} Q_i$ be the Cartesian product of the quandles $Q_i$. It is easy to verify that if the set
$$
Q=\{ (q_i) \in P~:~ \varphi_{ij}(q_j)=q_i ~\textrm{for all} ~i,j \in I ~\textrm{s.t}~ i \leq j\}
$$
is non-empty, then it is a subquandle of $P$ and is the projective limit of $(Q_i, \varphi_{ij}, I)$.

The projective limit of a projective system of a residually finite quandles is residually finite, see \cite[Corollary 3.5]{MR3981139}.
\begin{example}\label{exa:1.7}
Let $I$ be a directed set and $(Q_i)_{i \in I}$ be a family of quandles. For each $i \in I$, fix an element $q_i\in Q_i$. Now for $i , j \in I$ with $i \leq j$, define $\varphi_{ij}: Q_j \to Q_i$ as $\varphi_{ij}(Q_j)=q_i$ and $\varphi_{ii}=\id_{Q_i}$. Then $(Q_i, \varphi_{ij}, I)$ is a projective system with projective limit $Q=\{(q_i)\}$, which is a trivial quandle with one element.
\end{example}

\begin{example}\label{exa:1.8}
Let $Q$ be a quandle. Then $\Ker(Q)$ is a directed set under the reverse inclusion, that is, for $\alpha, \beta \in \Ker(Q)$, $\alpha \leq \beta$ if and only if $\alpha \supseteq \beta$. The family $(Q/\gamma)_{\gamma \in \Ker(Q)}$ along with the canonical quotient homomorphisms $\varphi_{\alpha, \beta}: Q/\beta \to Q/\alpha$, for all $\alpha, \beta \in \Ker(Q)$ with $\alpha \leq \beta$, give rise to a projective system of quandles $(Q/\alpha, \varphi_{\alpha \beta}, \Ker(Q))$. Define a homomorphism $\varphi: \underset{\gamma \in \Ker(Q)}{\varprojlim} \to Q$ as $([x]_{\gamma})_{\gamma \in \Ker(Q)} \to [x]_{\id_Q}.$ It is clear that the map $\varphi^{-1}:Q \to \underset{\gamma \in \Ker(Q)}{\varprojlim} Q/\gamma$ defined as $x \mapsto ([x]_{\gamma})_{\gamma \in \Ker(Q)}$ is the inverse of $\varphi$. Thus the projective limit of the system is $Q$ itself.
\end{example}

\begin{example}\label{exa:1.9}
Let $Q$ be a quandle. Then $\underset{< \infty}{\Ker(Q)}$ is a directed set under the reverse inclusion. The family $(Q/\gamma)_{\gamma \in \underset{< \infty}{\Ker(Q)}}$ along with the canonical quotient homomorphisms $\varphi_{\alpha, \beta}: Q/\beta \to Q/\alpha$, for all $\alpha, \beta \in \underset{< \infty}{\Ker(Q)}$ with $\alpha \leq \beta$, give rise to a projective system of quandles $(Q/\alpha, \varphi_{\alpha \beta}, \underset{< \infty}{\Ker(Q)})$. The projective limit of this system is called the {\it profinite completion} of the quandle $Q$ and is denoted by $\hat{Q}$. Note that there is a canonical homomorphism $\eta: Q \to \hat{Q}$ defined as $\eta(q) = ([q]_{\gamma})_{\gamma \in \underset{< \infty}{\Ker(Q)}}$ and the kernel of $\eta$ is the profinite kernel of $Q$, that is, $\ker(\eta)=\underset{< \infty}{\mathcal{K}(Q)}$.
\end{example}

If $Q$ is a finite quandle, then its profinite completion is $Q$ itself.

\begin{proposition}\label{pro:1.10}
Let $Q$ be a quandle and $\hat{Q}$ its profinite completion. Then $Q$ is residually finite if and only if the map $\eta: Q \to \hat{Q}$ is injective.
\end{proposition}
\begin{proof}
If $Q$ is residually finite, then by Proposition \ref{pro:1.6}, the profinite kernel of $Q$ is the diagonal on $Q \times Q$. Thus the map $\eta:Q \to \hat{Q}$ is injective. Conversely, suppose that $\eta$ is injective. Then for any distinct elements $p,q \in Q$, we have $\eta(p) \neq \eta(q)$, which implies that $(p,q) \not \in \underset{< \infty}{\mathcal{K}(Q)}$. By Proposition \ref{pro:1.6}, we get that $Q$ is residually finite.
\end{proof}

\begin{definition}
A quandle $Q$ is called {\it profinite} if it is the limit of some projective system of finite quandles.
\end{definition}

By definition, all finite quandles are profinite. From \cite[Corollary 3.5]{MR3981139}, we know that profinite quandles are residually finite.

\begin{example}\label{exa:functorial_properties}
If a group $G$ is the limit of projective system of groups $(G_i, \varphi_{ij}, I)$, then one can check that $\Conj_n(G)=  \underset{i \in I}{\varprojlim} \Conj_n(G_i)$, where $n \in \mathbb{Z}$, and $\Core(G)= \underset{i \in I}{\varprojlim} \Core(G_i)$. Thus if $G$ is a profinite group, then $\Conj_n(G)$ and $\Core(G)$ are profinite quandles.
\end{example}

\begin{example}\label{exa:projective_limit_of_trivial_quandles}
For each $n \in \mathbb{N}$, let $T_n=\{1,2, \ldots, n\}$ be the trivial quandle with $n$ elements. For $p,q \in \mathbb{N}$, such that $p \leq q$, define a map $\varphi_{p,q}: T_q \to T_p$ such that
\begin{align*}
\varphi_{p,q}(x)=\begin{cases}
x ~~\textrm{ if } x \leq p\\
p~~\textrm{ otherwise}
\end{cases}
\end{align*}
Clearly, $(T_q, \phi_{p,q}, \mathbb{N})$ forms a projective system of quandles. It is easy to see that its projective limit  is the set $\{ \overbar{x}_n, \dot{x}~|~ n\in \mathbb{N}\}$, which is the trivial quandle with countably many infinite elements. Here, $\overbar{x}_n$ denotes the sequence $(x_i)_{i \in \mathbb{N}}$, where 
\begin{align*}
\begin{cases}
x_i=i ~\textrm{ for } i \leq n\\
x_j = n ~ \textrm{ for } j \geq n
\end{cases},
\end{align*}
and $\dot{x}$ is the sequence $(\dot{x}_i)_{i \in \mathbb{N}}$,  with $\dot{x}_i=i$ for all $i \in \mathbb{N}$. Thus the trivial quandle with countably many infinite elements is a profinite quandle.
\end{example}

It is easy to observe that the class of profinite quandles is closed under taking finite Cartesian products. Thus there are non-trivial profinite quandles which are countably infinite.

It is a well-known fact that a countable profinite group is a finite group. However, Example \ref{exa:projective_limit_of_trivial_quandles} illustrates that this is not the case for quandles. Thus it would be interesting to study the profinite property in case of quandles.

\section{Topological characterization of profinite quandles}\label{sec:topological_characterization_of_profinite_quandles}
A {\it topological quandle} $Q$ is a quandle along with a topological structure under which the binary operation $* : Q \times Q \to Q$ is continuous, and the right multiplication map $R_q: Q \to Q$, mapping $p \mapsto p *q$, is a homeomorphism. A subspace $P$ of a topological quandle $Q$ is called a {\it subquandle} if it satisfies the condition that $a*b \in P$ for all $a, b \in P$, and for every $y, b \in P$, the mapping $y \mapsto y *b$ is a homeomorphism of $P$ onto itself.

A topological quandle is said to be {\it connected, Hausdorff, compact}, and {\it topologically disconnected} if the corresponding property holds for its underlying topological space.

Let $Q$ be a compact quandle. An equivalence relation $\alpha$ on $Q$ is said to be {\it open} ({\it closed}) if it is open (closed) in $Q \times Q$ under the product topology. If $\alpha$ is an open equivalence on $Q$, then the classes of $\alpha$ are both open and closed. A congruence $\alpha$ on $Q$ is open if and only if the quotient quandle $Q/\alpha$ is a finite quandle with the discrete topology.

Consider a topological quandle $Q$. For any element $a \in Q$, denote $Q^a$ as the connected component of $a$, which is the maximal connected subset of $Q$ containing $a$.
\begin{proposition}
Let $Q$ be a topological quandle and $a \in Q$ a fixed element. Then $Q^a$ is a subquandle of $Q$.
\end{proposition}
\begin{proof}
Let $x \in Q^a$. According to the definition of connected component, $Q^a= Q^x$. Since $R_x$ is a homeomorphism of $Q$ onto itself, the set $\{b *x ~:~ b \in Q^a\}$ is $Q^x$. Similarly $\{b *^{-1}x ~:~ b \in Q^a\}$ is $Q^x$. Thus $Q^a$ is a subquandle of $Q$. 
\end{proof}

A projective system of topological quandles is the projective system of quandles along with the requirement that all the homomorphisms are continuous.

The following proposition is easy to prove.

\begin{proposition}\label{prop:topology_on_projective_limit}
Suppose $(Q_i, \varphi_{ij}, I)$ is a projective system of topological quandles, where each $\varphi_{ij}$ is continuous and $Q$ its projective limit. Then $Q$ inherits a natural topological quandle structure as a subquandle of the product $P=\underset{i \in I}{\prod}~Q_i$, where $P$ has the product topology, and for each $i \in I$, the homomorphisms $\varphi_i:Q \to Q_i$ are continuous. Moreover, $Q$ is unique up to topological quandle isomorphism.
\end{proposition}

\begin{lemma}\label{lemma:projective_limit_is_closed_in_Hausdorff}
Let $(Q_i, \varphi_{ij},I)$ be a projective system of compact Hausdorff topological quandles. Then the projective limit $Q=\underset{i \in I}{\varprojlim}~{Q_i}$ is a closed subquandle of $P=\underset{i \in I}{\prod}~Q_i$. In particular, if for every $i \in I$, $Q_i$ is Hausdorff and compact, then $Q$ is Hausdorff and compact. 
\end{lemma}

\begin{proof}
Let $i,j \in I$ such that $i \leq j$. Consider the following continuous homomorphism:
\begin{center}
\begin{tikzcd}
P \arrow[r, "\pi_i \times \pi_j"] & Q_i \times Q_j \arrow[r, "\id_i \times \varphi_{ij}"] & Q_i \times Q_i,
\end{tikzcd}
\end{center}
where for each $i \in I$, $\pi_i: P \to Q_i$ is the projection map. Given that $Q_i$ is Hausdorff for each $i \in I$, we note that the diagonal $D$ on $Q_i \times Q_i$ is closed and consequently, the set $P_{i,j}=(\pi_i \times \pi_j)^{-1} \circ \big( (\id_i \times \varphi_{ij})^{-1} (D) \big)$ is a closed set in $P$. Note that $Q=\underset{i,j \in I, i \leq j}{\cap} {P_{ij}}$, and thus is a closed subset of $P$. Thus if all $Q_i$ are Hausdorff and compact, where $i \in I$, then $Q$ is Hausdorff and compact.
\end{proof}

\begin{theorem}\label{thm:existence_of_projective_limits}
Let $(Q_i, \varphi_{ij}, I)$ be a projective system of compact Hausdorff quandles. Then the projective limit $Q= \underset{i \in I}{\varprojlim}~ Q_i$ exists. Moreover, if the connecting homomorphism $\varphi_{ij}: Q_j \to Q_i$ are onto, then each compatible homomorphism $\varphi_i: Q \to Q_i$ is onto. 
\end{theorem}
\begin{proof}
Let 
$$
Q=\{ (q_i)_{i \in I} \in P~:~ \varphi_{ij}(q_j)=q_i ~\textrm{for all} ~i,j \in I ~\textrm{s.t}~ i \leq j\},
$$
where $P= \prod_{i \in I} Q_i$. By Lemma \ref{lemma:projective_limit_is_closed_in_Hausdorff}, $Q$ is compact. Now we will prove that $Q$ is non-empty. For each $k \in I$ consider the subquandle $R_k$ of $P$ defined as 
$$
R_k=\{ (q_i)_{i \in I} \in P~:~ \varphi_{jk}(q_k)=q_j ~\textrm{for all} ~j \in I ~\textrm{s.t}~ j \leq k\}.
$$
Clearly, $Q= \underset{i \in I}{\cap}R_i$. For $x \in Q_j$, and for each $h \leq j$, taking $q_h=\varphi_{hj}(x)$ and arbitrary otherwise, we see that there is an element $(q_i)_{i \in I} \in R_j$ such that $q_j=x$. Thus $R_j$ is non-empty. Moreover, for $j \leq j'$, observe that $R_j \supseteq R_{j'}$, and as $I$ is a directed set, it imply that the family of sets of the from $R_j$ has the finite intersection property. Furthermore, for each $k \in I$, one can prove that $R_k$ is closed using the same arguments used in Lemma \ref{lemma:projective_limit_is_closed_in_Hausdorff} proving that $Q$ is closed in $P$. Now since $P$ is compact, the projective limit $Q=\underset{i \in I}{\varprojlim} ~Q_i$ being the intersection of closed sets is non-empty and hence exists.

Now we will prove that the homomorphisms $\varphi_i: Q \to Q_i$  are onto. Note that for each $i \in I$, $\varphi_i $ is the restriction of the projection map $\pi_i: P \to Q_i$ to $Q$. Fix $j \in I$ and $x \in Q_j$. We claim that for every finite subset $F \subset I$, the set $\underset{i \in F}{\cap} (R_i \cap \pi_j^{-1}(x))$ is non-empty. Since $I$ is a directed set, pick an upper bound, $l$, of elements of $F \cup \{ j\}$. Since for every $i \in F$, $R_l \subseteq R_i$, to prove the claim it is sufficient to prove that $R_l \cap \pi_j^{-1}(x)$ is non-empty. Since $\varphi_{jl}: Q_l \to Q_j$ is onto, there exists $y \in Q_l$ such that $\varphi_{jl}(y)=x$. By the same argument in the last paragraph explaining the non-emptiness of $R_i$ sets, we see that there exists an element $q=(q_i) \in R_l$ such that $q_l =y$ and $\varphi_{jl}(y)=q_j=x$. Thus $q \in R_l \cap \pi_j^{-1}(x)$ and this proves the claim. Because $P$ is compact, and the family $\{ R_i, \pi_j^{-1}(y)~:~ i \in I \}$ has the finite intersection property, we observe that the set $\underset{i \in I}{\cap}(R_i \cap \pi_j^{-1}(x))=Q \cap \pi_j^{-1}(x)$ is non-empty. Consequently, the map $\varphi_j:Q \to Q_j$ is onto.
\end{proof}

\begin{corollary}\label{cor:profinite_quandles_are_compact_and_Hausdorff}
Let $(Q_i, \varphi_{ij}, I)$ be a projective system of finite quandles. Then the projective limit $\underset{i \in I}{\varprojlim} ~Q_i$ exists and is a compact Hausdorff topological quandle.
\end{corollary}
Thus profinite quandles are compact Hausdorff and totally disconnected. In case of group and semigroups the converse holds, that is, a compact Hausdorff and totally disconnected group (semigroup) is profinite. Thus the following question is natural to ask.

\begin{question}\label{ques:converse?}
Let $Q$ be a compact Hausdorff and totally disconnected quandle. Is $Q$ a profinite quandle?
\end{question}

The above result implies that for a given quandle $Q$, the profinite completion $\hat{Q}$ exists.

\begin{proposition}
Let $Q$ be a quandle and $\hat{Q}$ its profinite completion. Then the image of $Q$ under the canonical homomorphism $\eta: Q \to \hat{Q}$ is a dense subquandle.
\end{proposition}
\begin{proof}

Let $U$ be a non-empty open subset of $\hat{Q}$. Then $U=\hat{Q} \cap \prod_{\alpha \in \underset{< \infty}{\Ker(Q)}} U_{\alpha}$, where all but finitely many of the $U_{\alpha}$ are equal to $Q/ \alpha$. We will show that $U \cap \eta(Q)$ is non-empty. Let ${\alpha}_1, \ldots, {\alpha}_n \in \underset{< \infty}{\Ker(Q)}$ for which $U_{\alpha} \neq Q / \alpha$. Take $\beta = \alpha_1 \cap \cdots \cap \alpha_n \in \underset{< \infty}{\Ker(Q)}$. Let $V_{\alpha_1}, \ldots, V_{\alpha_n} \subset Q /\beta$ be the preimages of $U_{{\alpha}_i}$, where $1 \leq i \leq n$, under the onto maps $Q/ \beta \twoheadrightarrow Q / \alpha_i$. We claim that the set $V = \bigcap \limits_{ i=1} ^{n} V_{\alpha_i}$ is non-empty. Note that there is a projection map $\pi_{\beta}: U \to Q/ \beta$, thus by the definition of the projective limit $\pi_{\beta}( U) \subset V$. Thus $V$ is non-empty. For $x \in \pi_{\beta}( U )$, observe that $\eta(q_{\beta}^{-1}(x)) \in U$, where $q_{\beta}: Q \to Q /\beta$ is the quotient map. This completes the proof.
\end{proof}

\begin{proposition}\label{prop:closed_subquandle_of_a_profinite_quandle}
Let $Q$ be a profinite quandle and $P$ a closed subquandle of $Q$. Then $P$ is a profinite quandle.
\end{proposition}
\begin{proof}
Suppose $(Q_i, \varphi_{ij}, I)$ is a projective system of finite quandles whose projective limit is $Q$, where $\varphi_i:Q \to Q_i$ are the compatible homomorphisms. For each $i \in I$, denote $P_i= \varphi_i(P)$, a finite subquandle of $Q_i$. It is easy to see that $(P_i, {\varphi_{ij}}_{|P_j}, I)$ is a projective system and the maps ${\varphi_i}_{|P}: P \to P_i$ are homomorphisms compatible with it. By \cite[Proposition 1.1.6 (c)]{MR1691054}, $P$ is dense in $\underset{i \in I}{\varprojlim}~P_i$. Given that $P$ is closed in $\underset{i \in I}{\varprojlim} ~P_i$, it implies $P=\underset{i \in I}{\varprojlim} ~P_i$.
\end{proof}

A Hausdorff quandle $Q$ is said to be {\it topologically residually finite} if for any distinct elements $a, b \in Q$, there exists a finite quandle $F$ with the discrete topology and a continuous homomorphism $f: Q \to F$ such that $f(a) \neq f(b)$.

\begin{theorem}\label{thm:topological_charactersization_of_profinite_quandles}
The following are equivalent for a compact Hausdorff quandle.
\begin{enumerate}
\item $Q$ is profinite.
\item $Q$ is topologically residually finite.
\item $Q$ is closed subquandle of a direct product of finite quandles.
\end{enumerate}
\end{theorem}
\begin{proof}
For {\it (1)} $\implies$ {\it (2)}. Let $Q$ be a profinite quandle. Then there exists a projective system $(Q_i, \varphi_{ij}, I)$ of finite quandles with projective limit $Q$. Now for $(q_i) \neq (q'_i) \in Q$, there exists $j \in I$ such that $q_j \neq q_j'$. Thus $\varphi_j( (q_i) ) \neq \varphi_j( (q'_i) )$, where $\varphi_j : Q \to Q_j$ is the continuous compatible homomorphism. Hence $Q$ is residually finite as topological quandle.

For {\it (2)} $\implies$ {\it (3)}. Let $Q$ be a compact topologically residually finite quandle. For $p,q \in Q$ with $p \neq q$, there exists a finite quandle $F_{p,q}$ and a continuous homomorphism $\phi_{p,q}: Q \to F_{p,q}$ such that $\phi_{p,q}(p) \neq \phi_{p,q}(q)$. Now consider the quandle
$$
F = \underset{(p,q) \in Q \times Q, p \neq q}{\prod} F_{p,q},
$$
where $F$ is a compact Hausdorff quandle under the product topology. Define a homomorphism $\phi: Q \to F$ as
$$
\phi = \underset{(p,q) \in Q \times Q, p \neq q}{\prod} \phi_{p,q}
$$
which is an injective continuous map into a Hausdorff space. Thus $Q$ is isomorphic to a closed subquandle of direct product of finite quandles.

For {\it 3} $\implies$ {\it 1}. Let $P=\underset{i \in I}\prod F_i$ be a product quandle, where $I$ is an indexing set, and for each $i \in I$, $F_i$ is a finite quandle. Let $\mathcal{I}$ be the collection of finite subsets of $I$. Then $\mathcal{I}$ is a directed set under the inclusion ordering. For all $J_1, J_2 \in \mathcal{I}$, such that $J_1 \leq J_2$, consider the projection maps $\pi_{J_1J_2}: \underset{i \in J_2}{\prod}F_i \to \underset{i \in J_1}{\prod} F_i$. Then $P$ is the projective limit of $(\underset{i \in J}{\prod} F_i, \pi_{JK}, \mathcal{I})$, where $J \leq K \in \mathcal{I}$. Now the proof follows from Proposition \ref{prop:closed_subquandle_of_a_profinite_quandle}.
\end{proof}

\section{Endomorphism monoid of a profinite quandle}\label{sec:endo_of_profinite_monoid}
In this section, with $Q$ being a topological quandle, $\Aut(Q)$ and $\End(Q)$ refer to the sets of continuous automorphisms and continuous endomorphisms of $Q$, respectively.

For a compact Hausdorff quandle $Q$, there exists a unique uniform structure on $Q$ compatible with its topology, namely, the collection of all neighborhoods of the diagonal of $Q \times Q$. The collection of open equivalences form a fundamental system of entourages of the induced uniformity. As a result of the property that each continuous function from a compact Hausdorff space to a uniform space is uniformly continuous, the right multiplication on $Q$ is uniformly continuous and all continuous endomorphisms are uniformly continuous.

Given a profinite quandle $Q=\underset{i \in I}{\varprojlim}~Q_i$, where for each $i \in I$, $Q_i$ is a finite quandle, the induced topology on $Q$ is the coarsest collection of open sets that preserves the continuity of the compatible homomorphisms $\varphi_j: Q \to Q_j$ for all $j \in I$. Consequently, the collection of open congruences on $Q$ serves as a fundamental system of entourages for the induced uniformity on $Q$.

The following result provides a sufficient condition for a profinite quandle to have a fundamental system of entourages of open fully invariant congruences.

\begin{lemma}\label{lem:fundamental_system_of_open_fully_invariant_congruences}
Let $Q$ be a profinite quandle. If $Q$ has only a finite number of open congruences of index $n$ for each $n \geq 1$, then $Q$ has a fundamental system of open fully invariant congruences.
\end{lemma}
\begin{proof}
Let $\mathcal{A}_n$ be the collection of all open congruences on $Q$ of index at most $n$. Given that $\mathcal{A}_n$ is a finite set, the congruence, $\alpha_n= \underset{\alpha \in \mathcal{A}_n}\cap \alpha$, is open and is an entourage. According to the definition of uniform structure, the set $S=\{\alpha_n~:~n \geq 1\}$ is a fundamental system of entourages for the uniformity on $Q$. We will prove that every element of $S$ is fully invariant. Let $f:Q \to Q$ be a continuous homomorphism and $\beta \in \mathcal{A}_n$. Consider the quotient map $q_{\beta}: Q \to Q/\beta$. Note that $(f \times f)^{-1}(\beta)=\ker (q_{\beta} \circ f)$ and is therefore of index at most $n$. If $\alpha_n \in S$ and $(a,b) \in \alpha_n$, then $(a,b) \in \alpha$ for all $\alpha \in \mathcal{A}_n$, implying $(a,b) \in (f \times f)^{-1}(\alpha)$, which means $(f(a), f(b)) \in \alpha$ for all $\alpha \in \mathcal{A}_n$. Thus, $(f(a), f(b)) \in \alpha_n$. Thus $S$ is a fundamental system of open fully invariant congruences.
\end{proof}

A topological quandle $Q$ is said to be {\it topologically generated} by a subset $X$ if the algebraic subquandle $\bar{X}$ generated by $X$, which is the intersection of all subquandles of $Q$ containing $X$, is dense in $Q$. We say $Q$ is {\it topologically finitely generated} by $X$ if $X$ is a finite set.

\begin{corollary}\label{cor:topologically_finitely_generated_profinite_quandle_has_fundamental_system}
If $Q$ is topologically finitely generated profinite quandle, then it has a fundamental system of open fully invariant congruences.
\end{corollary}
\begin{proof}
Suppose $Q$ is topologically generated by a finite set $X$, and $\bar{X}$ denotes the subquandle of $Q$ algebraically generated by $X$. For a given finite quandle $F$ with the discrete topology, if $f, g: Q \to F$ are two continuous homomorphisms, and $f=g$ on $\bar{X}$, then $f=g$ on $Q$ (see \cite[Corollary 1, Page 76]{MR1726779}). Because $X$ and $F$ are finite sets, the number of homomorphisms from $\bar{X}$ to $F$ is finite, which in turn implies that there are only finitely many continuous homomorphism from $Q$ to $F$. Due to the finite number of quandles of order $n$, we conclude that there are only finitely many open congruences of order $n$.
\end{proof}

Given a profinite quandle $Q$ and an open fully invariant congruence $\alpha$ on $Q$, if $\phi: Q \to Q$ is a continuous endomorphism, then it induces an endomorphsim $\phi_{\alpha}: Q /\alpha \to Q/ \alpha$ defined as $\phi_{\alpha}([p]_{\alpha})=[\phi(p)]_{\alpha}.$

\begin{proposition}
Let $Q$ be a profinite quandle. If $Q$ has a fundamental system of open fully invariant congruences, then $Q$ is Hopfian.
\end{proposition}
\begin{proof}
Let $\phi: Q \to Q$ be a continuous epimorphism. We claim that $\phi$ is injective. To the contrary, suppose that there exist $p \neq q \in Q$ with $\phi(p) \neq \phi(q)$. Given that $Q$ is profinite, there exists an open congruence $\alpha$ on $Q$ such that $(p,q) \notin \alpha$. According to Lemma \ref{lem:fundamental_system_of_open_fully_invariant_congruences}, there exists an open fully invariant congruence $\rho \subset \alpha$. Consequently, there is an epimorphism $\phi': Q/\rho \to Q/\rho$. Since $Q/\rho$ is finite, thus $\phi'$ is an automorphism. Because $\phi(p) = \phi(q)$, it follows that $\phi'([p]_{\rho})=\phi'([q]_{\rho})$, which leads to a contradiction because $[p]_{\rho} \neq [q]_{\rho}$. Hence, $\phi$ is a bijective continuous map, and since $Q$ is compact and Hausdorff, it follows that $\phi$ is an automorphism. Thus, $Q$ is Hopfian
\end{proof}

Let $X$ and $Y$ be uniform spaces. Then a set $F$ of functions from $X$ to $Y$ is said to be {\it uniformly equicontinuous} if for any entourage $R \subset Y \times Y$, $\cap_{f \in F} (f \times f)^{-1}(R)$ is an entourage for $X$.

\begin{theorem}[Ascoli] \label{Ascoli}
Let $X$ and $Y$ be compact Hausdorff spaces equipped with their uniform structures and let $C(X,Y)$ the space of continuous maps from $X$ to $Y$ equipped with the compact-open topology. Then a set $F \subset C(X,Y)$ is compact if and only if $F$ is closed and uniformly equicontinuous.
\end{theorem}

\begin{proposition}\label{prop:closed_subset}
Let $P$ be a compact, Hausdorff quandle and $Q$ a Hausdorff quandle. Then $\Hom(P,Q)$ $($respectively, $\Aut(P)$$)$ is closed in compact-open topology.
\end{proposition}
\begin{proof}
This is the proof for $\Hom(P,Q)$. The proof for $\Aut(Q)$ can be obtained analogously by replacing $\Hom(P,Q)$ with $\Aut(P)$.

Consider a function $f \in C(P,Q) \setminus \Hom(P,Q)$. Then, for some $p,q \in P$, it holds that $f(p * q) \neq f(p) * f(q)$. Since $Q$ is Hausdorff, there exist disjoint open neighborhoods $U$ and $V$ of $f(p*q)$ and $f(p) * f(q)$, respectively. Because the binary operation in $Q$ is continuous, there exist open neighborhoods $W, W'$ of $f(p)$ and $f(q)$ such that $W*W' \subset V$. Let $\Omega\subset C(P,Q)$ be the set of all continuous maps $\phi: P \to Q$ satisfying $\phi(p*q) \in U$, $\phi(p) \in W$ and $\phi(q) \in W'$. Observe that

$$
\Omega = V(\{p\}, W) \cap V(\{q\}, W') \cap V(\{p*q\}, U),
$$
where $V(A,B)$ denotes the set of all continuous functions $f: P \to Q$ with $f(A) \subset B$. Since $P$ is Hausdorff and compact, thus finite subsets are compact in $P$, and thus $\Omega$ is open in the compact-open topology. Clearly, $\Omega$ is disjoint from $\Hom(P,Q)$ and $f \in \Omega$. This completes the proof.
\end{proof}

For a locally compact Hausdorff space $X$, it is a well-established result that the compact-open topology on $C(X,X)$ endows it with a topological monoid structure under the composition operation.

\begin{theorem}\label{thm:when_does_end_and_aut_are_profinite}
Suppose $Q$ is a compact Hausdorff and totally disconnected topological quandle. If $\End(Q)$ $($respectively, $\Aut(Q))$ is compact in the compact-open topology, then $\End(Q)$ $($respectively, $\Aut(Q))$ is a profinite monoid $($respectively, group). Furthermore, the compact-open topology coincides with the pointwise convergence topology.
\end{theorem}
\begin{proof}
Considering the pointwise convergence topology on $Q^Q$, the canonical inclusion map $i:\End(Q) \hookrightarrow Q^Q$ is continuous because the compact-open topology is finer than the pointwise convergence topology. Furthermore, since $\End(Q)$ is compact and $Q^Q$ is Hausdorff, the map $i$ is homeomorphism onto its image. Thus $\End(Q)$ is compact, Hausdorff and totally disconnected monoid. By \cite[Theorem 3.9.3]{MR4176669}, $\End(Q)$ is a profinite semigroup and noting the fact that if a profinite semigroup is a monoid, then it is a profinite monoid. Hence, we conclude that $\End(Q)$ is a profinite monoid.

The group $\Aut(Q)$ is compact (see Proposition \ref{prop:closed_subset}). Since it is a subspace of $\End(Q)$ it is also Hausdorff and totally disconnected. Thus it is a profinite group (see  \cite[Corollary 1.2.4]{MR1691054}). 
\end{proof}

\begin{remark}
Let $Q$ be a topological quandle. If $\End(Q)$ is profinite monoid, then $\Aut(Q)$ is a profinite group.
\end{remark}
\begin{theorem}\label{thm:compactness_of_end_and_aut_in_profinite_quandle}
Let $Q$ be a profinite quandle. Then $\End(Q)$ $($respectively, $\Aut(Q))$ is compact in the compact-open topology if and only if $Q$ admits a fundamental system of open fully invariant $($respectively, characteristics$)$ congruences.
\end{theorem}
\begin{proof}
We present the proof for $\End(Q)$, and the proof for $\Aut(Q)$ can be derived analogously by substituting `fully invariant' with `charactersitic'.

Assuming that $\End(Q)$ is compact, consider $\alpha$ as an open congruence on $Q$. By Ascoli's theorem, \ref{Ascoli}, $\End(Q)$ is uniformly equicontinuous, and thus the set
$$
\rho= \underset{\phi \in \End(Q)}{\cap} (\phi \times \phi)^{-1}(\alpha)
$$
is an entourage of the uniformity on $Q$. Obviously, $\rho$ is a congruence and according to the definition of the uniform structure on $Q$, it contains an open congruence on $Q$, implying that $\rho$ is an open congruence. Now, let $f \in \End(Q)$, then
$$
(f \times f)^{-1}(\rho)=\underset{\phi \in \End(Q)}{\cap} (\phi \circ f \times \phi \circ f)^{-1}(\alpha) \supset \underset{\phi \in \End(Q)}{\cap} (\phi \times \phi)^{-1}(\alpha) = \rho.
$$
Therefore, $\rho$ is an open fully invariant congruence on $Q$. Since the identity map on $Q$ is in $\End(Q)$, we conclude that $\rho \subset \alpha$. Thus the uniform structure on $Q$ admits a fundamental system of open fully invariant congruences.

Conversely, consider the case where the uniform structure on $Q$ admits a fundamental system of open fully invariant congruences. Let $\alpha$ be an entourage. Then there exist an open fully invariant congruence $\rho$ such that $\rho \subset \alpha$. This in turn implies that
$$
\underset{\phi \in \End(Q)}{\cap} (\phi \times \phi)^{-1} (\alpha) \supset \underset{\phi \in \End(Q)}{\cap} (\phi \times \phi)^{-1} (\rho) \supset \rho,
$$
and so by definition of uniform structures $\underset{\phi \in \End(Q)}{\cap} (\phi \times \phi)^{-1} (\alpha)$ is an entourage. Thus $\End(Q)$ is uniformly equicontinuous. Now by Proposition \ref{prop:closed_subset} and Ascoli's theorem \ref{Ascoli}, $\End(Q)$ is compact.
\end{proof}

\begin{corollary}
Let $Q$ be a topologically finitely generated profinite quandle. Then $\End(Q)$ is a profinite monoid, and $\Aut(Q)$ is a profinite group in the compact open topology, which coincides with the pointwise convergence topology.
\end{corollary}
\begin{proof}
Every profinite quandle, being a subquandle of Cartesian product of finite quandles with the discrete topology, is totally disconnected. The result now follows from Lemma \ref{lemma:projective_limit_is_closed_in_Hausdorff} and Theorem \ref{thm:compactness_of_end_and_aut_in_profinite_quandle} and Theorem \ref{thm:when_does_end_and_aut_are_profinite}.
\end{proof}

Like automorphism groups of profinite groups \cite[Proposition 4.4.3]{MR2599132}, given a profinite quandle $Q$ with a fundamental system of open fully congruences, the monoid $\End(Q)$ can be explicitly represented as a projective limit of finite monoids.

\begin{theorem}\label{thm:end_aut_inn_as_projective_limits}
Let $Q$ be a profinite quandle. Suppose $Q$ has a fundamental system of entourages $\mathcal{O}$ consisting of open fully invariant congruences. If $\alpha \in \mathcal{O}$, then the natural projection map $\Omega_{\alpha}: \End(Q) \to \End(Q/\alpha)$ is continuous and
$$
\End(Q) \cong \underset{\alpha \in \mathcal{O}}{\varprojlim} \Omega_{\alpha}(\End(Q)).
$$
The analogous results holds for $\Aut(Q)$ and $\Inn(Q)$ if there exists a fundamental system of open characteristic congruences for $Q$. Furthermore,
$$
\Inn(Q) \cong \underset{\alpha \in \mathcal{O_C}}{\varprojlim}{\Inn(Q/\alpha)},
$$
where $\mathcal{O_C}$ is the collection of open characteristic congruences for $Q$.
\end{theorem}
\begin{proof}
Let $\alpha$ be an open fully invariant congruence on $Q$ and $\phi \in \End(Q)$. Then,
\begin{align*}
\Omega_{\alpha}^{-1}(\Omega_{\alpha}(\phi))&=\{ \psi: Q \to Q~:~(\psi(q), \phi(q)) \in \alpha \textrm{ for all } q \in Q \}\\
&=\underset{q \in Q}{\cap} V( [q]_{\alpha}, [\phi(q)]_{\alpha}),
\end{align*}
where $V(A,B)$ denotes the set of all continuous functions $f: Q \to Q$ such that $f(A) \subset B$. For any $\alpha \in \mathcal{O}$, the  classes under $\alpha$ are both closed and open. Consequently, the sets $V([q]_{\alpha}, [\phi(q)]_{\alpha}) $ are open in the compact-open topology for each $q \in Q$. Moreover, since $Q$ is compact, there exist only finitely many classes under $\alpha$, and hence, the set $\Omega_{\alpha}^{-1}(\Omega_{\alpha}(\phi))$ is open. Thus $\Omega_{\alpha}$ is continuous.

If $\alpha, \gamma$ are open fully invariant congruences and $\gamma \subset \alpha$, then there is a canonical homomorphism $\Omega_{\gamma \alpha}: \Omega_{\gamma}(\End(Q)) \to \Omega_{\alpha}(\End(Q))$ defined as $\Omega_{\alpha \gamma}(\Omega_{\gamma}(\phi))= \Omega_{\alpha}(\phi)$. Consequently, we have the following commutative diagram:

\begin{center}
\begin{tikzcd}
                                                           & \Omega_{\gamma}(\End(Q)) \arrow[dd, "\Omega_{\alpha \gamma}"] \\
\End(Q) \arrow[ru, "\Omega_{\gamma}"] \arrow[rd, "\Omega_{\alpha}"'] &                                                    \\
                                                           & \Omega_{\alpha}(\End(Q))                                
\end{tikzcd}
\end{center}

Observe that the family $(\Omega_{\alpha}(\End(Q)), \Omega_{ \alpha \gamma}, \mathcal{O})$ forms a projective system of finite monoids, with $\mathcal{O}$ as the directed set under the reverse inclusion. According to \cite[Corollary 1.1.6]{MR2599132}, the family of continuous homomorphism $\{\Omega_{\alpha}\}_{\alpha \in \mathcal{O}}$ induces a continuous epimorphism
$$
\Omega: \End(Q) \to \underset{\alpha \in \mathcal{O}}{\varprojlim}~{\Omega_{\alpha}(\End(Q/\alpha)}.
$$
We now claim that $\Omega$ is injective. Suppose $\phi \neq \psi \in \End(Q)$, and let $q \in Q$ such that $\phi(q) \neq \psi(q)$. Then there exists an open fully invariant congruence $\alpha$ for which $(\phi(q), \psi(q)) \notin \alpha$. As a result, we deduce that $\Omega_{\alpha}(\phi) \neq \Omega_{\alpha}(\psi)$. Therefore, $\Omega$ is injective, and since $\Omega$ is a mapping between profinite spaces, it is a homeomorphism.

In case of $\Inn(Q)$ it is sufficient to note that if $f:Q \to P$ is an onto homomorphism, then it will induce onto group homomorphism $\Inn(Q) \twoheadrightarrow \Inn(P)$.
\end{proof}

\begin{ack}
The author is supported by the Fulbright-Nehru postdoctoral fellowship. The author would also like to thank the Department of Mathematics and Statistics at the University of South Florida for providing office space and resources.
\end{ack}

\bibliography{references1}{}
\bibliographystyle{abbrv}
\end{document}